\documentclass[11pt,oneside,english]{amsart}
\usepackage[T1]{fontenc}
\usepackage[latin9]{inputenc}
\usepackage{geometry}
\usepackage{xcolor}
\usepackage{pdfcolmk}
\usepackage{babel}
\usepackage{amsbsy}
\usepackage{amstext}
\usepackage{amsthm}
\usepackage{amssymb}
\usepackage{setspace}
\usepackage{esint}
\PassOptionsToPackage{normalem}{ulem}
\usepackage{ulem}
\usepackage[unicode=true,pdfusetitle,
 bookmarks=true,bookmarksnumbered=true,bookmarksopen=true,bookmarksopenlevel=1,
 breaklinks=false,pdfborder={0 0 1},backref=section,colorlinks=true]
 {hyperref}

\makeatletter

\numberwithin{equation}{section}
\numberwithin{figure}{section}
\theoremstyle{plain}
\newtheorem{thm}{\protect\theoremname}[section]
\theoremstyle{plain}
\newtheorem{conjecture}[thm]{\protect\conjecturename}
\theoremstyle{remark}

\theoremstyle{plain}
\newtheorem{lem}[thm]{\protect\lemmaname}

\usepackage{tikz}

\makeatother

\providecommand{\conjecturename}{Conjecture}
\providecommand{\lemmaname}{Lemma}
\providecommand{\remarkname}{Remark}
\providecommand{\theoremname}{Theorem}

\global\long\def\C{\mathbb{C}}%
\global\long\def\R{\mathbb{R}}%
\global\long\def\H{\mathcal{H}}%

\global\long\def\N{\mathbb{N}}%
\global\long\def\Z{\mathbb{Z}}%
\global\long\def\F{\mathbb{F}_{q}}%

\global\long\def\SL{\operatorname{SL}}%
\global\long\def\SO{\operatorname{SO}}%
\DeclareMathOperator{\supp}{supp}
\DeclareMathOperator{\spann}{span}
\DeclareMathOperator{\tr}{tr}
\DeclareMathOperator{\Sym}{Sym}

\global\long\def\n#1{\left\Vert #1\right\Vert }%

\begin{document}
\title{Optimal Lifting for the Projective Action of \texorpdfstring{$\operatorname{SL}_{3}\left(\mathbb{Z}\right)$}{SL3(Z)} 
}
\author{Amitay Kamber and Hagai Lavner}
\thanks{Amitay Kamber, amitay.kamber@gmail.com, \\
Hagai Lavner, hagai.lavner@mail.huji.ac.il,\\
Einstein Institute of Mathematics, The Hebrew University of Jerusalem}

\begin{abstract}
Let $\epsilon>0$ and let $q\to\infty$ be a prime. We prove that with high probability, given $x$, $y$ in the projective plane over $\mathbb{F}_{q}$ there exists $\gamma\in \SL_{3}\left(\mathbb{Z}\right)$, with coordinates bounded by $q^{1/3+\epsilon}$, whose projection to $\SL_{3}\left(\mathbb{F}_{q}\right)$ sends $x$ to $y$. The exponent $1/3$ is optimal and the result is a high rank generalization of Sarnak's optimal strong approximation theorem for $\SL_{2}\left(\mathbb{Z}\right)$.
\end{abstract}

\maketitle

\section{Introduction}

In his letter (\cite{sarnak2015lettermiller}), Sarnak proved the
following lifting theorem, which he called optimal strong approximation.
\begin{thm}
\label{thm:Sarnak}Let $\Gamma=\SL_{2}\left(\Z\right)$, $q\in \Z_{>0}$,  $G_{q}=\SL_{2}\left(\Z/q\Z\right)$
and let $\pi_{q}:\Gamma\to G_{q}$ be the quotient map. Then for every $\epsilon>0$, as $q\to\infty$, there exists a set $Y\subset G_{q}$ of size
$\left|Y\right|\ge\left|G_{q}\right|\left(1-o_{\epsilon}\left(1\right)\right)$,
such that for every $y\in Y$ there exists $\gamma\in\Gamma$ of norm
$\n{\gamma}_{\infty}\le q^{3/2+\epsilon}$, with $\pi_{q}\left(\gamma\right)=y$,
where $\n{\cdot}_{\infty}$ is the infinity norm on the coordinates
of the matrix.
\end{thm}

The exponent $3/2$ in Theorem~\ref{thm:Sarnak} is optimal, as the
the size of $G_{q}$ is asymptotic to $q^{3}$, while the number of
$\gamma\in \SL_{2}\left(\Z\right)$ satisfying $\n{\gamma}_{\infty}\le T$
grows asymptotically like the Haar measure of the ball $B_{T}$ of
radius $T$ in $\SL_{2}\left(\R\right)$ (\cite{duke1993density,maucourant2007homogeneous}),
i.e., $\mu\left(B_{T}\right)\asymp T^{2}$.

We use the standard notation $x\ll_{z}y$ to say that there is a constant $C$ depending only on $z$ such that $x\le Cy$, and $x\asymp_{z}y$ means that $x\ll_z y$ and $y\ll_z x$.

We wish to discuss extensions of this theorem to $\SL_{3}$, with a
view towards general $\SL_{N}$. If $\Gamma=\SL_{N}\left(\Z\right)$,
then the number of $\gamma\in\Gamma$ of satisfying $\n{\gamma}_{\infty}\le T$
also grows like the Haar measure of the ball of radius $T$ in $\SL_{N}\left(\R\right)$,
i.e., $\mu\left(B_{T}\right)\asymp T^{N^{2}-N}$ (\cite{duke1993density,maucourant2007homogeneous}),
while the size of $G_{q}=\SL_{N}\left(\Z/q\Z\right)$ is $\left|G_{q}\right|\asymp q^{N^{2}-1}$.
One is therefore led to the following:
\begin{conjecture}
\label{conj:Sarnak}Let $\Gamma=\SL_{N}\left(\Z\right)$, $q\in \Z_{>0}$, $G_{q}=\SL_{N}\left(\Z/q\Z\right)$
and let $\pi_{q}:\Gamma\to G_{q}$ be the quotient map. Then for every $\epsilon>0$, as $q\to\infty$, there exists a set $Y\subset G_{q}$ of size
$\left|Y\right|\ge\left|G_{q}\right|\left(1-o_{\epsilon}\left(1\right)\right)$,
such that for every $y\in Y$ there exists $\gamma\in\Gamma$ of norm
$\n{\gamma}_{\infty}\le q^{\left(N^{2}-1\right)/\left(N^{2}-N\right)+\epsilon}$,
with $\pi_{q}\left(\gamma\right)=y$, where $\n{\cdot}_{\infty}$
is the infinity norm on the coordinates of the matrix.
\end{conjecture}

While we were unable to prove Conjecture~\ref{conj:Sarnak} even
for $N=3$, we prove a similar theorem for a non-principal congruence
subgroup of $\SL_{3}\left(\Z\right)$. For a prime $q$, let $\F$ be the field with $q$ elements, let $P_{q}=P^{2}\left(\F\right)$
be the 2-dimensional projective space over $\F$, i.e., the set of
vectors $\left(\begin{smallmatrix}
a \\
b \\
c
\end{smallmatrix}\right)$, $a, b, c\in\F$ not all $0$, modulo the equivalence relation $\left(\begin{smallmatrix}
a\\
b\\
c
\end{smallmatrix}\right)\sim\left(\begin{smallmatrix}
\alpha a\\
\alpha b\\
\alpha c
\end{smallmatrix}\right)$ for  $\alpha\in\F^\times$.  
The group $\SL_{3}\left(\F\right)$
acts naturally on $P_{q}$, and by composing this action with $\pi_{q}$
we have an action $\Phi_{q}\colon\SL_{3}\left(\Z\right)\to \Sym\left(P_{q}\right)$.

\begin{thm}\label{thm:almost radius}Let $\Gamma=\SL_{3}\left(\Z\right)$, and
for a prime $q$ let $P_{q}=P^{2}\left(\F\right)$ and $\Phi_{q}:\SL_{3}\left(\Z\right)\to\Sym\left(P_{q}\right)$ as above. Then for every $\epsilon>0$, as $q\to\infty$, there exists a set $Y\subset P_q$ of size $\left|Y\right|\ge\left(1-o_{\epsilon}(1)\right)\left|P_{q}\right|$,
such that for every $x\in Y$, there exists a set $Z_{x}\subset P_{q}$
of size $\left|Z_{x}\right|\ge\left(1-o_{\epsilon}(1)\right)\left|P_{q}\right|$,
such that for every $y\in Z_{x}$, there exists an element $\gamma\in\Gamma$
satisfying $\n{\gamma}_{\infty}\le q^{1/3+\epsilon}$, such
that $\Phi_{q}\left(\gamma\right)x=y$.
\end{thm}
The exponent $1/3$ is optimal, since the size of $P_{q}$ is
$\left|P_{q}\right|\asymp q^{2}$, while the number of elements $\gamma\in \SL_{3}\left(\Z\right)$
satisfying $\n{\gamma}_{\infty}\le T$ is $\asymp T^{6}$. 

An alternative formulation of Theorem~\ref{thm:almost radius} is that for all but $o_\epsilon(|P_q|^2)$ of pairs $(x,y)\in P_q\times P_q$, there exists an element $\gamma \in \Gamma$ satisfying $\n{\gamma}_\infty\le q^{1/3+\epsilon}$ such that $\Phi_{q}\left(\gamma\right)x=y$. However, in this formulation it is a bit harder to see why the exponent $1/3$ is optimal, and our proof actually uses the formulation of Theorem~\ref{thm:almost radius} as stated.

An important observation is that the premise of Theorem~\ref{thm:almost radius}
actually fails for the point
$x = \boldsymbol{1}=\left(\begin{smallmatrix}
0\\
0\\
1
\end{smallmatrix}\right)\in P_{q}$.
Elements sending $\boldsymbol{1}$ to $\left(\begin{smallmatrix}
a\\
b\\
c
\end{smallmatrix}\right)\in P_{q}$ necessarily have the third column modulo $q$ equivalent to $\left(\begin{smallmatrix}
a\\
b\\
c
\end{smallmatrix}\right)$ (modulo the action of $\F^\times$). Since there are only
$\asymp T^{3}$ possibilities for the third column, we need to consider matrices
of infinity norm at least $q^{2/3}$ in order to reach from $x=\boldsymbol{1}$
to almost all of $y\in P_{q}$. As a matter of fact, one may use the
explicit property (T) of $\SL_{3}\left(\R\right)$ from \cite{oh2002uniform}
together with ideas from \cite{ghosh2014best} to deduce that if we
allow the size of the matrices to reach $q^{2/3+\epsilon}$ we may
replace the set $Y$ in Theorem \ref{thm:almost radius} by the entire
set $P_{q}$.

We deduce Theorem~\ref{thm:almost radius} from a lattice point counting
argument, in the spirit of the work of Sarnak and Xue (\cite{sarnak1991bounds}).
To state it, we first define a different gauge of largeness on $\SL_{3}\left(\Z\right)$
by $\n{\gamma}_{\infty}\|\gamma^{-1}\|_{\infty}$. The number of $\gamma\in \SL_{3}(\Z)$
satisfying $\n{\gamma}_{\infty}\|\gamma^{-1}\|_{\infty}\le T$ grows
asymptotically like $T^2\log(T)$ (\cite{maucourant2007homogeneous}). Note that if $\n{\gamma}_{\infty}\le T$
then $\n{\gamma^{-1}}_{\infty}\le2T^{2}$. In particular, the ball
of radius $2T$ relatively to $\n{\cdot}_{\infty}\|\cdot^{-1}\|_{\infty}$
contains the ball of radius $T^{1/3}$ relatively to $\n{\cdot}_{\infty}$,
and their volume is asymptotically the same up to $T^{o(1)}$.
The counting result is as follows:

\begin{thm}
\label{thm:lattice point}Let $\Gamma=\SL_{3}\left(\Z\right)$, and
for a prime $q$ let $P_{q}=P^{2}\left(\F\right)$ and $\Phi_{q}:\SL_{3}\left(\Z\right)\to\Sym\left(P_{q}\right)$
as above. Then there exists a constant $C>0$ such that for every prime
$q$, $T\le Cq^{2}$ and $\epsilon>0$ it holds that 
\[
\left|\left\{ \left(\gamma,x\right)\in \SL_{3}\left(\Z\right)\times P^{2}\left(\F\right):\n{\gamma}_{\infty}\|\gamma^{-1}\|_{\infty}\le T,\Phi_{q}\left(\gamma\right)\left(x\right)=x\right\} \right|\ll_{\epsilon}q^{2+\epsilon}T.
\]
\end{thm}

Underlying Conjecture~\ref{conj:Sarnak} is the principal congruence
subgroup $\Gamma\left(q\right)=\ker\pi_{q}$. 
Let $\boldsymbol{1}=\left(\begin{smallmatrix}
0\\
0\\
1
\end{smallmatrix}\right)\in P_{q}$. Then the group
\[
\Gamma_{0}^{\prime}\left(q\right)=\left\{ \gamma\in \SL_{3}\left(\Z\right):\Phi_{q}\left(\gamma\right)(\boldsymbol{1})=\boldsymbol{1}\right\} =\left\{ \left(\begin{array}{ccc}
* & * & a\\
* & * & b\\
* & * & *
\end{array}\right)\in \SL_{3}\left(\Z\right):a=b=0\mod q\right\} 
\]
is a non-principal congruence subgroup of $\SL_{3}\left(\Z\right)$.
Theorem~\ref{thm:almost radius} says that Conjecture~\ref{conj:Sarnak}
holds ``on average'' for the non-principal congruence subgroup $\Gamma_{0}^{\prime}\left(q\right)$.

Conjecturally, such ``optimal lifting on average'' should hold for every sequence of congruence subgroups of $\Gamma = \SL_N(\Z)$, i.e., subgroups of some $\Gamma(q)$, $q>1$ an integer. We provide a further example of this phenomena for the action of $\SL_3(\Z)$ on flags of $\F^3$ in Theorem~\ref{thm:almost radius_flags}. 

Let us provide a spectral context for our results, namely Sarnak's density conjecture for exceptional eigenvalues. See also \cite{golubev2020sarnak} for a more detailed discussion.

Theorem~\ref{thm:Sarnak} follows from Selberg's conjecture about
the smallest non-trivial eigenvalue of the Laplacian of the hyperbolic
surfaces $\Gamma\left(q\right)\backslash\H$, where $\H$ is the hyperbolic plane and $\Gamma\left(q\right)$ is the $q$-th principal congruence subgroup of $\Gamma=\SL_{2}\left(\Z\right)$.
While Selberg's conjecture is widely open, Sarnak proved Theorem~\ref{thm:Sarnak}
using density estimates on exceptional eigenvalues of the Laplacian,
which are due to Huxley (\cite{huxley1986exceptional}). Similar density
results were proved by Sarnak and Xue using lattice point counting
arguments in \cite{sarnak1991bounds}, but only for arithmetic quotients
which are compact. The compactness assumption was removed in \cite{huntley1993density,gamburd2002spectral}
(and the results were moreover extended to some thin subgroups of
$\SL_{2}\left(\Z\right)$). As a matter of fact, in rank $1$ the density property is equivalent to the lattice point counting property (\cite{golubev2020sarnak}). 

In higher rank, Conjecture~\ref{conj:Sarnak} would similarly follow 
from a naive Ramanujan conjecture for $\Gamma\left(q\right)\backslash \SL_{N}\left(\R\right)$,
$\Gamma=\SL_{N}\left(\Z\right)$, which says (falsely!) that the representation
of $\SL_{N}\left(\R\right)$ on $L^{2}\left(\Gamma\left(q\right)\backslash \SL_{N}\left(\R\right)\right)$
decomposes into a trivial representation and a tempered representation.
The Burger-Li-Sarnak explanation of the failure of the naive Ramanujan
conjecture (\cite{burger1992ramanujan}) is closely related to the
behavior of the point $x_{0}=\boldsymbol{1}\in P_{q}$. 
As in rank $1$, Theorem~\ref{thm:lattice point} should be equivalent to density estimates for $\Gamma_{0}^{\prime}\left(q\right)$, but there are some technical problems coming from the fact that $\SL_3\left(\mathbb{Z}\right)$ is not cocompact (\cite{golubev2020sarnak}). Closely related density
results were recently proven by Blomer, Buttcane and
Maga for $N=3$ in \cite{blomer2017applications}, and for general $N$ by Blomer in \cite{blomer2019density}, using the Kuznetsov trace formula,
and it is very likely that Theorem~\ref{thm:almost radius} can
also be proven (and generalized to $N>3$) using those density arguments. However, the results of \cite{blomer2017applications} and \cite{blomer2019density}, concern cusp forms, and one has to deal with the presence of non-tempered Eisenstein representations and some other technical issues. Our counting approach is more elementary, and allows simpler generalizations, such as Theorem~\ref{thm:almost radius_flags}.

\subsection*{Structure of the article}
We provide a proof of Theorem~\ref{thm:Sarnak} in Section~\ref{sec:SL2}, which serves as a guideline for the harder case of $\SL_3$. The main difference between our proof and the proof in \cite{sarnak2015lettermiller} is that we avoid using spectral decomposition, which is far harder in $\SL_3$. 

In Section~\ref{sec:First Section} we prove Theorem~\ref{thm:lattice point} . The proof uses basic number theory and linear algebra. 

In Section~\ref{sec:Second Section} we deduce Theorem~\ref{thm:almost radius} from Theorem~\ref{thm:lattice point}. The argument is analytic, and uses various tools from spectral analysis and representation theory, which include property (T), the pre-trace formula (in a disguised form), and bounds on Harish-Chandra's $\Xi$ function. This section is based on a general framework developed by the first author with Konstantin Golubev surrounding similar questions (\cite{golubev2020sarnak}).

Finally, in Section~\ref{sec:flag action} we prove Theorem~\ref{thm:almost radius_flags} which is a variant of Theorem~\ref{thm:almost radius} for the action of $\SL_3(\Z)$ on flags of $\F^3$.

\subsection*{Acknowledgments}
We are grateful to Amos Nevo and Elon Lindenstrauss for various discussions surrounding this project, and for Peter Sarnak for his continued encouragement. We also thank Amos Nevo for his careful reading of a previous version of this article and for pointing out some inaccuracies.

This work is part of the Ph.D. thesis of the both authors at the Hebrew University of Jerusalem. The first author is under the guidance of Prof.~Alex Lubotzky, and is supported by the ERC grant 692854. The second author is under the guidance of Prof.~Tamar Ziegler, and is supported by the ERC grant 682150.

\section{\label{sec:SL2}Proof of Theorem~\ref{thm:Sarnak}}

The basic input for the proof of  Theorem~\ref{thm:Sarnak} is the following counting result, proved in \cite[Lemma~5.3]{gamburd2002spectral} (it also appeared earlier, e.g., in \cite{huxley1986exceptional}).
\begin{thm} \label{thm:counting SL2} Let $\epsilon>0$. Then for every $q\in \N$, the size of the set 
\[\left\lbrace \gamma \in \SL_2 (\Z): \gamma = I \mod q, \n{\gamma}_\infty \le T\right\rbrace\]
is bounded by 
$\ll_\epsilon T^\epsilon(T^2/q^3+T/q+1)$.
\end{thm}
\begin{proof}
Let $\gamma = \begin{pmatrix}
a & b \\
c & d
\end{pmatrix}
\in \SL_2(\Z)$ be in the set. 
It holds that $\gamma - I \in q M_n(\Z)$, so $\det(\gamma -I)=0 \mod q^2$, or explicitly
\[
(a-1)(d-1)-bc = 0 \mod q^2. 
\]
Since $ad-bc = 1$, we have $a+d = 2 \mod q^2$. Since both $a$ and $d$ are bounded in absolute value by $T$, the number of options for $a+d$ is at most $4T/q^2 +1$. Similarly, the number of options for $a$ is at most $2T/q+1$. Therefore, the number of options for $(a,d)$ is $\ll (T/q^2+1)(T/q+1)$. 

To determine $b,c$, note that if $ad\ne 1$, then $bc = 1-ad\ne 0$, and by standard divisor bounds this gives $\ll_\epsilon T^\epsilon$ options for $(b,c)$. Otherwise, assuming $q>2$, $a=d=1$, and then either $b=0$ or $c=0$. If $b=0$ then $c$ has at most $2T/q+1$ options, while if $c=0$, then $b$ has at most $2T/q+1$ options.

All in all, the number of solution is bounded by 
\[
\ll_\epsilon (T/q^2+1)(T/q+1)T^\epsilon + T/q +1
\ll T^\epsilon (T^2 /q^3 +T/q+1).
\]

\end{proof}

Our proof of Theorem~\ref{thm:Sarnak} proceeds with some spectral analysis of hyperbolic surfaces associated to $\SL_2(\Z)$ and its congruence subgroups, which will require some preliminaries. 
Let $\H$ be the hyperbolic plane, with the model $\H = \left\lbrace z= x+iy \in \C: y>0\right\rbrace$. The space $\H$ is equipped with the metric defined by $d(x+iy,x'+iy')=\operatorname{arcosh}\left(1+ \frac{(x-x')^2+(y-y')^2}{2yy'}\right)$ and a measure defined by $\frac{dxdy}{y^2}$. It also has a natural $\SL_2(\R)$ action by M\"obius transformation, i.e., $\left(\begin{smallmatrix}
a & b \\
c & d
\end{smallmatrix}\right)
z = \frac{az + b}{cz+d}
$.

This action allows us to identify $\H$ with $G/K$, where $G = \SL_2(\R)$, and $K=\SO(2)$ is the stabilizer of the point $i\in\H$. We also assume that the Haar measure on $G$ is normalized to agree with the measure on $\H$ on right $K$-invariant measurable sets.

When using spectral arguments, it will be useful to use a bi-$K$-invariant (i.e., left and right $K$-invariant) gauge of largeness of an element. We therefore define $\n{g}_\H= e^{ d(i,g i)/2}$. Explicitly, by the Cartan decomposition of $G$, $g$ can be written as 
\[g = k_1 
\begin{pmatrix}
e^{r/2} &\\
 & e^{-r/2}
\end{pmatrix}
k_2,
\]
with $k_1,k_2 \in K= \SO(2)$, and $r \in \R_{\ge0}$ unique. Then $\n{g}_\H =e^{r/2}$. As the $L^2$-norm of the coordinates of $\gamma$ is $\sqrt{e^r+e^{-r}}$, $\n{g}_\H$ is closely related to the infinity norm on the coordinates, namely, there exists a constant $C>0$ such that 
$C^{-1} \n{g}_\infty \le \n{g}_\H \le C\n{g}_\infty$. We may therefore prove Theorem~\ref{thm:Sarnak} using the gauge $\n{\cdot}_\H$ instead of $\n{\cdot}_\infty$. Two important properties of $\n{\cdot}_\H$ are symmetry $\n{g}_\H=\n{g^{-1}}_\H$, and sub-multiplicativity $\n{g_1g_2}_\H \le \n{g_1}_\H \n{g_2}_\H$. The sub-multiplicativity follows from the fact that $d$ is a $G$-invariant metric on $\H$.

We define the function $\chi_T \in L^1 \left(K\backslash G / K\right)$ as the normalized probability characteristic function of the set $\left\lbrace g \in G \colon \n{g}_\H \le T \right\rbrace$, i.e.,
\[
\chi_T(g) = 
\frac{1}{2\pi (\cosh{2\log(T)}-1)}
\begin{cases}
1 & \n{g}_\H \le T\\
0 & \n{g}_\H >T\\
\end{cases}.
\]
Notice that $2\pi (\cosh{r}-1)$ is the volume of the hyperbolic ball of radius $r$. Here and later by a probability function we mean a non-negative function with integral $1$.

We also define $\psi_T \in L^1 \left(K\backslash G / K\right)$ as the function 
$\psi_T(g) = \frac{1}{T}
\begin{cases}
\n{g}_\H^{-1} & \n{g}_\H \le T\\
0 & \n{g}_\H >T\\
\end{cases}.
$

There is a convolution of $f\in L^\infty (G/K) \cong L^\infty\left(\H\right)$ and $\chi \in L^1\left(K\backslash G /K\right)$, which we usually think as an action of $\chi$ on $f$. It is simply the convolution of the two functions, when both are considered as invariant functions on $G$:

\begin{align*}
f*\chi (x) &= \intop_{g\in G} f(x g^{-1}) \chi(g) dg \\
& = \intop_{g\in G} f(g^{-1}) \chi(g x) dg
\end{align*}

It holds that $f*\chi\in L^\infty(\H)$. %If $f\in L^1(K \backslash G / K)$, then $f*\chi \in L^1(K \backslash G / K)$ is also defined. 
For example, the value of $f*\chi_T$ at $g_0$, is the average of $f$ over the ball $\left\{g_0 g\in G : \n{g}_\H \le T\right\}$.

\begin{lem}[Convolution Lemma]\label{lem:ConvolutionLemmaSL2} For every $g\in G$,
$(\chi_T * \chi_T)(g) \ll \psi_{T^2}(g)$.
\end{lem}
We refer to  \cite[Lemma 2.1]{sarnak1991bounds} or \cite[Proposition~5.1]{gamburd2002spectral} for a proof. Geometrically, the proof calculates the volume of an intersection of two hyperbolic balls. In Lemma~\ref{lem:convolution lemma} we give a spectral proof of a similar statement for $\SL_3(\R)$, which also works for $\SL_2(\R)$, but adds a factor that is logarithmic in $T$.

As in the statement Theorem~\ref{thm:Sarnak}, let $q\in \Z_{>0}$, $\Gamma = \SL_2(\Z)$,  $G_{q}=\SL_{2}\left(\Z/q\Z\right)$
and let $\pi_{q}:\Gamma \to G_{q}$ be the quotient map. Let $\Gamma(q) = \ker{\pi_q}$.

We look at the locally symmetric space $X_q := \Gamma(q) \backslash \H \cong \Gamma \backslash G/ K$. This space is a hyperbolic orbifold of finite volume. By $L^2(X_q)$ we mean the Hilbert space of measurable functions on $X_q$ with bounded $L^2$-norm relative to the finite measure on $X_q$, with the obvious inner-product. We still consider a function on $X_q =\Gamma(q) \backslash \H = \Gamma(q) \backslash G/K$ as a left $\Gamma(q)$-invariant function on $\H$ or on $G$. Right convolution by functions from $L^1(K \backslash G/K)$ is defined for bounded functions on $X_q$, and extends to functions in $L^2(X_q)$ as the convolution defines a bounded operator. In particular, we will consider right convolution of $f\in L^2(X_q)$ with $\chi_T$.

For $x_0\in X_q $, denote $b_{T,x_0}(x) := \sum_{\gamma \in \Gamma(q)} \chi_T\left( \tilde{x}_0^{-1} \gamma x\right)$, when $\tilde{x}_0$ is any lift of $x_0$ to $G$. It holds that $b_{T,x_0} \in L^2(X_q)$, and $\intop_{X_q}b_{T,x_0}(x) dx =1$. 

In particular $b_{T,e}$ corresponds to the point $\Gamma(q)e K \in \Gamma(q)\backslash \H$, where $e$ is the identity matrix in $G$.

\begin{lem}\label{lem:unfolding SL2}
For $f\in L^2(X_q)$ bounded,
\[
\left\langle f, b_{T,x_0} \right\rangle = f* \chi_{T} (x_0).
\]
\end{lem}

\begin{proof}
By unfolding,
\begin{align*}
    \left\langle f, b_{T,x_0} \right\rangle & = \intop_{x\in \Gamma(q) \backslash \H} f(x) \sum_{\gamma\in \Gamma(q)} \chi_T\left(x_0^{-1}\gamma x\right) dx \\
    & = \intop_{x\in \Gamma(q) \backslash \H}
    \sum_{\gamma\in \Gamma(q)}
    f(\gamma x)  \chi_T\left(x_0^{-1}\gamma x\right) dx \\
    & = \intop_{x\in \H} f(x) \chi_T\left(x_0^{-1}x\right)dx \\
    & = \intop_{x\in \H} f(x) \chi_T\left(x^{-1}x_0\right)dx \\
    & = f*\chi_T (x_0).
\end{align*}
Notice that we used the fact that $\chi_T(g) = \chi_T\left( g^{-1} \right)$, which is a simplification that will not occur in $\SL_3$.
\end{proof}

The following lemma uses the combinatorial  Theorem~\ref{thm:counting SL2} to get analytic information:
\begin{lem}\label{lem:lattice count to L2 - SL2}
It holds that 
\[\n{b_{T,e}}_2^2 \ll_\epsilon T^\epsilon \left(\frac{1}{q^3} + \frac{1}{T^2}\right).
\]

In particular, for $T=q^{3/2}$, 
\[
\n{b_{T,e}}_2^2 \ll_\epsilon \frac{T^\epsilon}{q^3}.
\]
\end{lem}

\begin{proof} 
By Lemma~\ref{lem:unfolding SL2}, and Lemma~\ref{lem:ConvolutionLemmaSL2},
\begin{align*}
\n{b_{T,e}}_2^2 & = b_{T,e} *\chi_T (e) \\
& = \sum_{\gamma\in\Gamma(q)} (\chi_T *\chi_T)(\gamma) \\
& \ll \sum_{\gamma\in\Gamma(q)} \psi_{T^2}(\gamma) \\
    & =\frac{1}{T^2} 
    \sum_{\gamma\in \Gamma(q): \n{\gamma}_\H\le T^2} \n{\gamma}_\H^{-1}.
\end{align*}

We next apply discrete partial summation (\cite[Theorem~421]{hardy1979introduction}) which says that for $g:\Gamma\left(q\right)\to\left[1,\infty\right]$,
$f:\left[1,\infty\right]\to\R$ nice enough it holds that 
\begin{align}\label{eq:partial summation}
\sum_{\gamma:1\le g\left(\gamma\right)\le Y}f\left(g\left(\gamma\right)\right)=f\left(Y\right)|\left\{ \gamma:1\le g\left(\gamma\right)\le Y\right\}| -\intop_{1}^{Y}|\left\{ \gamma:g\left(\gamma\right)\le S\right\}|\frac{df}{dS}\left(S\right)dS.
\end{align}

Apply this to $g(\gamma) = \n{\gamma}_\H$, $f(x) = x^{-1}$ and $Y=T^2$, 

\begin{align*}
    &\frac{1}{T^2} 
    \sum_{\gamma\in \Gamma(q): \n{\gamma}_\H\le T^2} \n{\gamma}_\H^{-1} \\
    & =    \frac{1}{T^2} 
 \left( \frac{1}{T^2}|\left\{\gamma\in \Gamma(q):\n{\gamma}_\H \le T^2 \right\}|
    + \intop_1^{T^2} |\left\{\gamma\in \Gamma(q):\n{\gamma}_\H \le S \right\}| S^{-2} dS
    \right) \\
    &\ll_\epsilon T^\epsilon\frac{1}{T^2}\left(\frac{1}{T^2}\left(\frac{T^4}{q^3} + \frac{T^2}{q}+1\right)
    +
    \intop_1^{T^2}
    \frac{1}{S^2}
    \left(\frac{S^2}{q^3}+\frac{S}{q} +1\right)  dS \right)
    \\
    &\ll_\epsilon T^\epsilon \frac{1}{T^2} \left(\frac{T^2}{q^3} + \frac{1}{q} +\frac{1}{T^2} + 1\right) \\
    &\ll T^\epsilon \left(\frac{1}{q^3} + \frac{1}{T^2}\right).
\end{align*}

The first inequality follows from Theorem~\ref{thm:counting SL2}. 

\end{proof}

Let $\pi \in L^2(X_q)$ be the constant probability function on $X_q$ (recall that the space has finite volume). Denote by $L^2_0\left(X_q\right)$ the set of functions of integral $0$, or alternatively the set of functions orthogonal to $\pi$. The deepest input to the proof is the following celebrated theorem of Selberg:

\begin{thm}[Selberg's Spectral Gap Theorem]\label{thm:selberg spectral gap}
There is an explicit $\tau>0$ such that for every $f\in L^2_0\left(X_q\right)$ and $T>0$ is holds that 
$\n{f*\chi_{T^\eta}}_2 \ll T^{-\eta\tau}\n{f}_2 $.
\end{thm}

The important part of the theorem is the independence of $\tau$ from $q$.

Selberg's theorem is usually stated as a lower bound on the spectrum of the Laplacian. However, it is well known that it can be translated to a spectral gap of the convolution operators by large balls (see, e.g., \cite[Section~4]{ghosh2013diophantine} for a generalized statement). There are various results improving the value of $\tau$ in Selberg's theorem (see \cite{sarnak2005notes}), but those improvement are inconsequential to our theorem.

From Selberg's theorem we deduce:

\begin{lem}\label{lem:almost equidist L2 SL2}
For $T= q^{3/2}$,
\[
\n{b_{T,e}*\chi_{T^\eta}-\pi}_2 \ll_\epsilon q^{-3/2-\eta\tau+\epsilon}.
\]
\end{lem}
\begin{proof}
We have $b_{T,e}-\pi\ \in L^2_0(X_q)$ and $\pi *\chi_{T} = \pi$ (as an average of the constant function is the constant function).

Therefore,
\begin{align*}
\n{b_{T,e}*\chi_{T^\eta}-\pi}_2
& =
\n{(b_{T,e}-\pi)*\chi_{T^\eta}}_2 \\
& \ll T^{-\eta \tau}\n{b_{T,e}-\pi}_2 \\
& \ll_\epsilon q^{-3/2-\eta\tau+\epsilon}, 
\end{align*}
where in the first inequality we applied Theorem~\ref{thm:selberg spectral gap}, and in the second inequality we applied $\n{b_{T,e}-\pi}_2 \le \n{b_{T,e}}_2$ ($b_{T,e}-\pi$ is the orthogonal projection of $b_{T,e}$ onto $L^2_0(X_q)$) and Lemma~\ref{lem:lattice count to L2 - SL2}.
\end{proof}

The last lemma implies that the function $b_{T,e}*\chi_{T^\eta}$ is very close to the constant probability function $\pi$. Let us show how this implies Theorem~\ref{thm:Sarnak}.

We have a map $\iota\colon G_q\cong \Gamma(q) \backslash \Gamma \to X_q \cong \Gamma(q) \backslash G /K$, defined as $\iota(\Gamma(q) \gamma) = \Gamma(q) \gamma K$. For $y\in G_q$, we may consider the function $b_{T_0,\iota(y)}$. We choose $T_0$ small enough (independently of $q$), so that the functions $b_{T_0,\iota(y)}$ will have disjoint supports for $\iota(y) \ne \iota(y^\prime)$. Specifically, it is enough to choose $T_0$ such that the ball of radius $2\log(T_0)$ around $i$ and around $\gamma i \ne i$ for $\gamma \in \SL_2(\Z)$ are disjoint.  We also notice that $\iota$ has fibers of bounded size, specifically $|\SL_2(\Z) \cap K| =4$. In addition, there is a uniform (in $q$) upper bound on the norm $\n{b_{T_0,\iota(y)}}_2$.

\begin{lem}\label{lem:lifting from spectral SL2} Assume that $\left\langle b_{T,e}*\chi_{T^\eta},b_{T_0,\iota(y)}\right\rangle>0$, then there is $\gamma \in \Gamma$ such that $\pi_q(\gamma)=y$, and $\n{\gamma}_\H \le T_0T^{1+\eta}$.
\end{lem}
\begin{proof}
By Lemma~\ref{lem:unfolding SL2}, the condition implies that 
\[
(b_{T,e}*\chi_{T^\eta}*\chi_{T_0}) (\iota(y)) >0.
\]

Treat the function as a left $\Gamma(q)$-invariant and right $K$-invariant function on $G$. Let $\gamma_y$ to be a lift of $y$ to $\Gamma$, i.e. $\pi_q(\gamma_y) = y$. Therefore, $b_{T,e}*\chi_{T^\eta}*\chi_{T_0} (\gamma_y) >0$.

By the definition of convolution, there are $g_1^\prime,g_2,g_3\in G$, such that $g_1^\prime \in \supp(b_{T,e})$, $g_2 \in \supp(\chi_{T^\eta})$, $g_3 \in \supp(\chi_{T_0})$, and such that $ g_1^\prime g_2 g_3 = \gamma_y$. Looking at the definition of $b_{T,e}$ and $g_1^\prime$, there are $g_1\in \supp(\chi_T)$, $\gamma\in \Gamma(q)$ such that $e^{-1} \gamma g_1^\prime = g_1$ (we write $e$ for the identity element instead of discarding it, anticipating the case of $\SL_3$ below). Therefore $\gamma^{-1} e g_1 g_2 g_3 = \gamma_y$ .

Write $g = g_1 g_2 g_3$. By the above, $\n{g}_\H \le \n{g_1}_\H \n{g_2}_\H \n{g_3}_\H \le T_0T^{1+\eta}$. In addition, $ e g = \gamma \gamma_y$, so that $g \in \Gamma(q) \gamma_y$. Therefore $g\in \Gamma$ and $\pi_q(g) = y$, as needed.

\end{proof}

We may now finish the proof of Theorem~\ref{thm:Sarnak}. Let $\eta>0$ and write $T=q^{3/2}$. Assume that $Z \subset G_q$ is the set of $y\in G_q$ such such that there is no $\gamma_y \in \Gamma$ with $\n{\gamma_y} \le T_0 T^{1+\eta}$ and $\pi_q(\gamma_y)=y$. It suffices to prove that for a fixed $\eta>0$ it holds that $|Z|= o(q^3)$. 

By Lemma~\ref{lem:lifting from spectral SL2}, for every $y\in Z$, 
\[
\left \langle b_{T,e}*\chi_{T^\eta},b_{T_0,\iota(y)} 
\right \rangle
=0.
\]

Let $B = \sum_{y\in Z} b_{T_0,\iota(y)}$. Then by the above and the fact that $\left \langle \pi,b_{T_0,\iota(y)} \right \rangle = \frac{1}{\text{Vol}(\Gamma(q) \backslash \H)}\gg \frac{1}{q^3}$, 
\[
\left|\left \langle b_{T,e}*\chi_{T^\eta}-\pi,B
\right \rangle \right|
\gg \frac{|Z|}{q^3}.
\]

On the other hand, by the choice of $T_0$ and the remarks following it, $\n{B}_2^2 \ll |Z|$. Therefore, using Lemma~\ref{lem:almost equidist L2 SL2} and Cauchy-Schwartz,
\begin{align*}
    & \left|\left \langle b_{T,e}*\chi_{T^\eta}-\pi,B
\right \rangle\right| \ll \n{B}_2 \n{b_{T,e}*\chi_{T^\eta}-\pi}_2 \\
& \ll_\epsilon  \sqrt{|Z|} q^{-3/2-\eta\tau+\epsilon}.
\end{align*}

Combining the two estimates and taking $\epsilon$ small enough gives
\[
|Z| \ll_\epsilon q^{3-2\eta\tau+2\epsilon}=o(q^3),
\]
as needed.

\section{\label{sec:First Section}Proof of Theorem~\ref{thm:lattice point}}

Our goal is to prove that there exists a constant $C>0$ such that for every prime $q$, $\epsilon>0$ and $T\le Cq^{2}$,
we have
\[
|\left\{ \left(\gamma,x\right)\in \SL_{3}\left(\Z\right)\times P^{2}\left(\F\right):\n{\gamma}_{\infty}\|\gamma^{-1}\|_{\infty}\le T,\Phi_{q}\left(\gamma\right)x=x\right\}| \ll_{\epsilon}Tq^{2+\epsilon}.
\]

If $\gamma\mod q$ has no eigenspace of dimension $2$, then it has
at most 3 eigenvectors in $P^{2}\left(\F\right)$. Call such a $\gamma$
\emph{good} mod $q$ and otherwise call it \emph{bad} mod $q$. Therefore for $T\le q^{2}$,
\begin{align*}
&|\left\{ \left(\gamma,x\right)\in \SL_{3}\left(\Z\right)\times P^{2}\left(\F\right):\n{\gamma}_{\infty}\|\gamma^{-1}\|_{\infty}\le T,\Phi_{q}\left(\gamma\right)x=x,\gamma\text{ good mod }q\right\}| \\
&\ll T^{2+\epsilon}\ll Tq^{2+\epsilon}.
\end{align*}

We therefore need to bound the number of bad $\gamma$-s. Notice that bad elements do exist and may have a lot of fixed points: e.g., the element $I\in \SL_{3}\left(\Z\right)$
is bad mod $q$ and $\Phi_{q}\left(I\right)$ fixes all of $P^{2}\left(\F\right)$. 

Assuming that we choose $C<1/4$, it will hold that either
$\n{\gamma}_{\infty}<q/2$ or $\n{\gamma^{-1}}_{\infty}<q/2$. Therefore
if $\gamma\ne I$ then $\gamma\mod q\ne I_{\SL_{3}\left(\F\right)}$,
and thus $\Phi_{q}(\gamma)$ fixes at most $q+1$ elements in $P^{2}\left(\F\right)$. 
It thus suffices to prove that for some $C>0$, and $T\le Cq^{2}$,
\[
|\left\{ \gamma\in \SL_{3}\left(\Z\right):\n{\gamma}_{\infty}\|\gamma^{-1}\|_{\infty}\le T\text{, }\gamma\text{ bad mod }q\right\}| \ll_{\epsilon}Tq^{1+\epsilon}.
\]

Assume that $\gamma$ is bad mod $q$ and $\n{\gamma}_{\infty}\|\gamma^{-1}\|_{\infty}\le T$.
Without loss of generality assume that $\n{\gamma}_{\infty}\le \n{\gamma^{-1}}_{\infty}\le T^{1/2}<q/2$.
We identify elements of $\F$ with integers of absolute value at most $q/2$. Thus, once we know the value of an entry of $\gamma \mod q$ we know the same entry in $\gamma$.

We divide the range of $\n{\gamma}_{\infty}$ into $O\left(\log\left(T\right)\right)$
dyadic subintervals. Denote by $S$ the bound on $\n{\gamma}_\infty$ and by $R$ the bound on $\n{\gamma^{-1}}_\infty$. Then it is enough to prove that there exists $C>0$ such that for
every $RS\le Cq^{2}$ and $S\le R$ it holds that
\[
|\left\{ \gamma\in \SL_{3}\left(\Z\right):\n{\gamma}_{\infty}\le S,\n{\gamma^{-1}}_\infty\le R\text{, }\gamma\text{ bad mod }q\right\}| \ll_{\epsilon}RSq^{1+\epsilon}.
\]
%Notice that we may also assume that $R\le 2 S^2$, since $\n{\gamma^{-1}}_\infty \le 2\n{\gamma}_\infty^2$ for every $\gamma \in \SL_3(\R)$, but we will not use it.

It will be useful to understand the behavior of bad $\gamma$. Let $\alpha\in\mathbb{F}_{q}\backslash\{0\}$ be the eigenvalue of
$\gamma\mod q$ with an eigenspace of dimension $2$. Then the third
eigenvalue is $\alpha^{-2}\mod q$.

From this it follows that 
$(\gamma - \alpha I)(\gamma - \alpha^{-2}I) = 0\mod q$,
or, 
\begin{equation}\label{eq: relation mod q}
\gamma +\alpha ^{-1}\gamma^{-1} = \alpha + \alpha^{-2} \mod q.
\end{equation}

By considering the trace of $\gamma$ and $\gamma^{-1}$ we have that 
\begin{equation}\label{eq: trace relation}
\tr \gamma = \alpha + 2\alpha^{-2} \mod q, \quad
\tr \gamma^{-1} = \alpha^{-1} + 2\alpha^2 \mod q.
\end{equation}

Finally, identify $\alpha$ with some lift of it in $\Z$. Then $\gamma -\alpha I \mod q$ is of rank $1$, which means that $\det(\gamma-\alpha I) = 0 \mod q^2$. Since $\det\gamma =1$, it holds that $\det(\gamma - xI) = 1 - \tr \gamma^{-1} x +\tr \gamma x^2 - x^3$, and we get 
\begin{equation}\label{eq: trace relation mod q2}
\alpha^2 \tr \gamma - \alpha \tr \gamma^{-1}= \alpha^3 -1 \mod q^2. 
\end{equation}

Denote the entries of $\gamma$ by $a_{ij}$, $1\le i,j\le3$ and the entries of $\gamma^{-1}$ by $b_{ij}$, $1\le i,j\le3$.

There are $\le(2S+1)^{3}$ options for choosing the diagonal $a_{11},a_{22},a_{33}$ of $\gamma$, and once we know them, we know $\tr \gamma$. 
By Equation~\eqref{eq: trace relation} $\alpha$ is a root of a known third degree polynomial, so there are at most $3$ options for $\alpha$. By Equation~\eqref{eq: trace relation mod q2} we know $\tr
\gamma^{-1} \mod q^2$. Since $R \le RS \le C q^2<q^2/4$,  we may assume that $|\tr \gamma^{-1}| < q^2/2$, so now we know $\tr \gamma^{-1}.$

By Equation~\eqref{eq: relation mod q} we now know the diagonal $b_{11},b_{22},b_{33} \mod q$ of $\gamma^{-1} \mod q$. Since the entries $b_{11},b_{22},b_{33}$ are bounded in absolute value by $R$, we have at most $2R/q+1$ options for each of them. We may guess $b_{11},b_{22}$ and get $b_{33}$ since we know $\tr \gamma^{-1}$. 

In total, we had $\ll S^{3}(R/q+1)^{2}$ options so far. 
We call the case where $a_{ii}a_{jj}= b_{kk}$ for some $\{i,j,k\}=\{1,2,3\}$ exceptional. We will deal with it later and assume for now that we are in the non-exceptional case. 

Notice that $a_{11}a_{22}-a_{12}a_{21} = b_{33}$, or 
\[
a_{12}a_{21} = a_{11}a_{22}-b_{33}.
\]
Since we are in the non-exceptional case, the right hand side is not $0$. By the divisor bound there are at most $\ll_\epsilon q^\epsilon$ options for $a_{12},a_{21}$. Similarly, all the other entries $a_{13},a_{31},a_{23},a_{32}$ have at most $\ll_\epsilon q^\epsilon$ options. 

In total, we counted $\ll_{\epsilon}q^{\epsilon}S^{3}(R/q+1)^{2}$
bad $\gamma$-s in the non-exceptional case. We postpone the exceptional
case to the end of the proof. The same (and better) bounds hold for
it as well.

It remains to show that
\[
S^{3}(R/q+1)^{2} \ll RSq,
\]
assuming $S\le R$, $RS \le C q^2$.

If $R\le q$, then we need to show that $S^3 \ll RSq$, or $S^2 \ll Rq$, which is obvious since $S\le R \le q$. 

If $R>q$ then we need to show that 
$S^3R^2/q^2\ll RS q$, or $S^2R \ll q^3$. Since $RS \le Cq^2$, this reduces to showing that $S\ll q$, which is obvious since $S^2 \le RS \le Cq^2$. 

\subsection*{Exceptional cases}

Recall that the exceptional case is when $a_{ii}a_{jj}=b_{kk}$ for some $\{i,j,k\} = \{1,2,3\}$. 
Assume without loss of generality that $a_{11}a_{22}=b_{33}$.
Therefore $a_{12}a_{21}=a_{11}a_{22}-b_{33}=0$.

We know that $\gamma-\alpha I\mod q$ is of rank $1$, so each determinant of a  $2\times2$
submatrix of $\gamma$ equals $0\mod q$. Therefore
\[(a_{11}-\alpha)(a_{22}-\alpha)-a_{12}a_{21}=0\mod q,\] 
so 
\[
(a_{11}-\alpha)(a_{22}-\alpha) =0 \mod q
\]

Without loss of generality again, we may assume that $a_{11} = \alpha \mod q$. 
By our assumptions
on the size of the matrix, we may lift $\alpha$ to some fixed element in $\Z$ of absolute value $\le q/2$ and let $a_{11}=\alpha$. By the above, 
$a_{12}a_{21}=0$, and by symmetry again, we may assume that $a_{21}=0$.
Some more minors give: 
\begin{align}
a_{31} (a_{22}-\alpha) & =a_{21} a_{32}=0\mod q\label{eq:excep21}\\
a_{31} a_{23} & =a_{21} (a_{33}-\alpha)=0\mod q.\label{eq:excep22}
\end{align}
We now divide into two cases according to whether $a_{31}=0$ or not:
\begin{enumerate}
\item \textbf{Case 1:} $a_{11}=\alpha$, $a_{21}=0$, $a_{31}=0$. In this
case, the matrix is of the form:
\[
\gamma=\left(\begin{array}{ccc}
\alpha & a_{12} & a_{13}\\
0 & a_{22} & a_{23}\\
0 & a_{32} & a_{33}
\end{array}\right).
\]
Denote $A=\left(\begin{array}{cc}
a_{22} & a_{23}\\
a_{32} & a_{33}
\end{array}\right)$. It holds that $\alpha\det A=1$. Therefore $\alpha=\pm1$ and $\det A=\pm1$.
We also know that the eigenvalues of $A\mod q$ are either $\pm1$
(if $\alpha=-1$) or $1$ with multiplicity 2 (if $\alpha = 1$). Therefore the trace
of $A$ is either $0$ or $2$. We now separate into two further cases. In
the first case $a_{22}\ne\alpha$ and $a_{33}\ne\alpha$, or equivalently $a_{22}a_{33}\ne \det A$. In the
second case we may assume without loss of generality that $a_{22}=\alpha$.
\begin{enumerate}
\item \textbf{Subcase 1a:} $a_{11}=\alpha$, $a_{21}=0$, $a_{31}=0,a_{22}\ne\alpha,a_{33}\ne\alpha$.
The entry $a_{22}$ has $2S+1$ options, and it determines the value of $a_{33}$
since we know the trace of $A$. In this subcase it holds that $a_{23}a_{32}=\det A -a_{22}a_{33}\ne0$.
By the divisor bound there are $\ll_{\epsilon}S^{\epsilon}$ options
for $a_{23},a_{32}$ and both are non-zero. We also know that the
third column of $\gamma -\alpha I \mod q$ is a multiple of the second column, and
now we know the ratio. This means that after we choose $a_{12}$
in $2S+1$ ways it sets $a_{13}$ uniquely. Therefore there are $\ll_{\epsilon}S^{2+\epsilon}\le RSq^\epsilon$ options in this case.
\item \textbf{Subcase 1b:} $a_{11}=\alpha,a_{21}=0,a_{31}=0,a_{22}=\alpha,a_{33}=1$.
In this case $a_{23}a_{32}=\det A -a_{22}a_{33}=0$. If $a_{23}\ne0$
then $a_{32}=a_{12}=0$ and there are $\le(2S+1)^2$ options for $a_{23},a_{13}$.
Similarly, if $a_{32}\ne0$ then $a_{23}=0$ and once we know $a_{12}$
we also know $a_{13}$. Therefore there are $\ll S^{2}\le RS$ option in this
case.
\end{enumerate}
\item \textbf{Case 2:} $a_{11}=\alpha$, $a_{21}=0$, $a_{31}\ne0$. By \eqref{eq:excep21}, \eqref{eq:excep22}
we have $a_{22}=\alpha,a_{23}=0$, and hence: 
\[
\gamma-\alpha I=\left(\begin{array}{ccc}
0 & a_{12} & a_{13}\\
0 & 0 & 0\\
a_{31} & a_{32} & a_{33}-\alpha
\end{array}\right)
\]
Since its rank mod $q$ is $1$ and $a_{31}\neq0$ the second and third
columns are scalar multiples of the first, thus $a_{12}=a_{13}=0$.
Therefore $\gamma$ is of the form 
\[
\gamma=\left(\begin{array}{ccc}
\alpha & 0 & 0\\
0 & \alpha & 0\\
a_{31} & a_{32} & a_{33}
\end{array}\right).
\]
Since $\det\gamma=1$ it holds that $\alpha=\pm1,a_{33}=1$ and there are $\ll S^{2}\le RS$ options for $\gamma$.
\end{enumerate}

\section{\label{sec:Second Section}Proof of Theorem~\ref{thm:almost radius}}

As in the proof of Theorem~\ref{thm:Sarnak}, the proof of Theorem~\ref{thm:almost radius} is analytic, and employs the combinatorial Theorem~\ref{thm:lattice point} as an input. Since we wish to use the usual notations of
dividing $\SL_{3}\left(\R\right)$ by $\SL_{3}\left(\Z\right)$ from
the left, we apply a transpose to the question as stated in Theorem~\ref{thm:almost radius}.

Let 
\[
\Gamma_{0}\left(q\right)=\left\{ \left(\begin{array}{ccc}
* & * & *\\
* & * & *\\
a & b & *
\end{array}\right)\in \SL_{3}\left(\Z\right):a=b=0\mod q\right\}.
\]
We have a right action of $\Gamma = \SL_{3}\left(\Z\right)$ on $\Gamma_{0}(q)$.
We let $P_{q}^{tr}=\Gamma_{0}\left(q\right)\backslash \Gamma $ (it is obviously isomorphic to $P_q$ as a set with a $\Gamma$ action).
Then Theorem~\ref{thm:almost radius} can be stated in the following
equivalent formulation:
\begin{thm}
\label{thm:main thm - alt version}As $q\to\infty$ among primes, for every $\epsilon>0$
there exists a set $Y\subset\Gamma_{0}\left(q\right)\backslash \Gamma =P_{q}^{tr}$
of size $\left|Y\right|\ge\left(1-o_{\epsilon}(1)\right)\left|P_{q}^{tr}\right|$,
such that for every $x_{0}\in Y$, there
exists a set $Z_{x_{0}}\subset P_{q}^{tr}$ of size $\left|Z_{x_{0}}\right|\ge\left(1-o_{\epsilon}(1)\right)\left|P_{q}^{tr}\right|$,
such that for every $y\in Z_{x_{0}}$,
there exists an element $\gamma\in \Gamma$ satisfying
$\n{\gamma}_{\infty}\le q^{1/3+\epsilon}$, such that $x_0\gamma=y$.
\end{thm}

Let $K=\SO(3)$ be the maximal compact subgroup of $G=\SL_3(\R)$.  By the Cartan
decomposition each element $g\in G$ can be written
as
\[
g=k_{1}\left(\begin{array}{ccc}
a_{1}\\
 & a_{2}\\
 &  & a_{3}
\end{array}\right)k_{2},
\]
with $k_{1},k_{2}\in \SO\left(3\right)$, and unique $a_{1},a_{2},a_{3}\in\R_{>0}$, satisfying $a_{1}\ge a_{2}\ge a_{3}>0$ and $a_{1}a_{2}a_{3}=1$. Define $\n{g}_K=a_{1}$.
Since $K=\SO\left(3\right)$ is compact there exists a constant $C>0$
such that
\[
C^{-1}\n g_{\infty}\le\n g_K\le C\n g_{\infty}.
\]
We may therefore prove Theorem~\ref{thm:main thm - alt version} using $\n{\cdot}_K$ instead of $\n{\cdot}_\infty$.

The size $\n{\cdot}_K$ will play the same role as $\n{\cdot}_\H$ in the $\SL_2$ case. Let us note some of its properties.
There is a constant $C>0$ such that $\n{g_1 g_2}_K \le C \n{g_1}_K\n{g_2}_K$ (actually, one may take $C=1$, but this detail will not influence us). A big difference from the $\SL_2$ case comes from the fact that $\n{\gamma}_K$ and $\n{\gamma^{-1}}_K$ can be quite different. However, it does hold that $\n{\gamma}_K \ll \n{\gamma^{-1}}_K^2$.

It will also be useful to define another bi-$K$ invariant gauge of largeness, by 
$\n{g}_\delta = a_1 a_3^{-1}$, where $a_1,a_3$ are as in the Cartan decomposition. It holds that there is a constant $C>0$ such that 
\begin{equation}
C^{-1} \n g_{\infty}\|g^{-1}\|_{\infty}\le\n{g}_\delta\le C\n g_{\infty}\|g^{-1}\|_{\infty}.\label{eq:equivalent distances}
\end{equation}

Now we have $\n{g}_\delta = \n{g^{-1}}_\delta$, and there is $C>0$ (which may be chosen to be $C=1$ by extra analysis) such that $\n{g_1 g_2}_\delta \le C \n{g_1}_\delta\n{g_2}_\delta$.

The relation between the two sizes is that $\n{g}_\delta \le \n{g}_K^3$, which follows from the fact that in the Cartan decomposition $a_3^{-1} = a_1a_2 \le a_1^2$, so $a_1 a_3^{-1} \le a_1^3$.

We will want to estimate the size of balls relative to $\n{\cdot}_K$ and $\n{\cdot}_\delta$. For this, we use the following formula for the Haar measure $\mu$ of $G$ (\cite[Proposition 5.28]{knapp2016representation}), which holds up to multiplication by a scalar:
\[
\intop_G f(g) d\mu = \intop_K \intop_K \intop_{\mathfrak{a}_+} f(k \exp(a) k') S(a) dk dk' da,
\]
where 
\[
\mathfrak{a}_+ = \left\{a = \left(\begin{array}{ccc}
\alpha_{1}\\
 & \alpha_{2}\\
 &  & \alpha_{3}
\end{array}\right) \in M_3(\R): \alpha_1\ge \alpha_2 \ge \alpha_3 , \alpha_1+\alpha_2+\alpha_3=0\right\},
\]
 and 
\[
S(a) = \sinh(\alpha_1-\alpha_2) \sinh(\alpha_2 - \alpha_3) \sinh(\alpha_3 - \alpha_1).
\]

Notice that for $\alpha_1-\alpha_2,\alpha_2-\alpha_3$ large, $S(a)$ behaves like $\n{a}_\delta^{2}$. This implies that 
\[
\mu\left(\left\{ g\in G:\n g_K\le T\right\}\right) \asymp T^6,
\]
and 
\[
\mu\left(\left\{ g\in G:\n g_\delta\le T\right\}\right) \asymp \log(T)T^2.
\]
See also \cite{maucourant2007homogeneous} for more accurate similar statements.

Let  $\chi_{T},\chi_{T,\delta} \in L^1(K \backslash G/K)$ be
\begin{align*}
\chi_{T}\left(g\right)&=\frac{1}{\mu\left(\left\{ g\in G:\n g_K\le T\right\}\right)}\begin{cases}
1 & \n g_K\le T\\
0 & \text{else}
\end{cases}, \\
\chi_{T,\delta}\left(g\right)&=\frac{1}{\mu\left(\left\{ g\in G:\n g_\delta\le T\right\}\right)}\begin{cases}
1 & \n{g}_\delta\le T\\
0 & \text{else}
\end{cases}.
\end{align*}
The functions $\chi_T,\chi_{T,\delta}$ are simply the probability characteristic functions of the balls according to $\n{\cdot}_K$ and $\n{\cdot}_\delta$.

Notice that for every $g\in G$, 
\[
\chi_T(g) \gg \log(T) \chi_{T^3,\delta} (g). 
\]

Let $\psi_{T}:G\to\R$ be 
\[
\psi_{T}\left(g\right)=\frac{1}{T}\begin{cases}
\n{g}_\delta^{-1} & \n{g}_\delta \le T\\
0 & \text{else}
\end{cases}.
\]

For $f: G \to \C$, we let $f^*:G\to \C$ be the function $f^*(g) = \overline{f\left(g^{-1}\right)}$.

Now we have the following version of Lemma~\ref{lem:ConvolutionLemmaSL2}:
\begin{lem}[Convolution Lemma]
\label{lem:convolution lemma}There exists a constant $C>0$ such
that for $T \ge 1$
\[
\chi_{T,\delta}*\chi_{T,\delta}\left(g\right)\le\left(\log\left(T\right)+2\right)^{C}\psi_{CT^{2}}\left(g\right).
\]
As a result, there exist a constant $C'>0$ such that for  $T \ge 1$
\[
\chi_{T}*\chi_{T}^{*}\le\left(\log\left(T\right)+2\right)^{C'}\psi_{C'T^{6}}\left(g\right).
\]
\end{lem}

\begin{proof}
Normalize $K$ to have measure 1. Let $\Xi:G\to\R_{+}$ be
Harish-Chandra's function, defined as
\[
\Xi\left(g\right)=\intop_{K}\delta^{-1/2}\left(gk\right)dk,
\]
where $\delta:G\to\R_{>0}$ is defined, using the Iwasawa decomposition
$G=KP$, as 
\[
\delta\left(k\left(\begin{array}{ccc}
a_{1} & * & *\\
0 & a_{2} & *\\
0 & 0 & a_{3}
\end{array}\right)\right)=a_{1}^{2}a_{3}^{-2}.
\]
(When restricted to $P$, $\delta$ is the modular function of $P$. Notice the similarity between $\delta(g)$ and $\n{g}_\delta^2$, hence the notation).

There are standard bounds on $\Xi$, given by (see, e.g., \cite[2.1]{trombi1972asymptotic})

\begin{equation}\label{eq:Xi bounds}
\n{g}_\delta^{-1}\le\Xi\left(g\right)\ll\left(\log\n{g}+1\right)^{C_{0}}\n{g}_\delta^{-1}    
\end{equation}

for some $C_{0}>0$. Using these upper bounds, we find that for some $C_{2}>0$,
\[
\intop_{G}\chi_{T,\delta}\Xi\left(g\right)dg=\frac{1}{\mu\left(\left\{ g\in G:\n g_\delta\le T\right\}\right)}\intop_{g:\n{g}_\delta\le T}\Xi\left(g\right)dg\ll\left(\log\left(T\right)+1\right)^{C_{2}}T^{-1}.
\]

Harish-Chandra's function $\Xi$ arises as follows (see, e.g., \cite[Section 3]{ghosh2013diophantine}). Let $(\pi,V)$ be the spherical representation of $G$ unitarily induced from the trivial character of $P$. It holds that if $f\in L^1(K\backslash G/K)$ and $v\in V$ is $K$-invariant, then 
\[
\pi(f)v=\intop_{G}f(g)\pi(g)vdg  = \left(  \intop_{G}f(g)\Xi(g)dg\right) v.
\]

Since $\pi(f_1*f_2)v=\pi(f_1)\pi(f_2)v$, 
\begin{align*}
\intop_{G}\left(\chi_{T,\delta}*\chi_{T,\delta}\right)\left(g\right)\Xi\left(g\right)dg & =\left(\intop_{G}\chi_{T,\delta}\left(g\right)\Xi\left(g\right)dg\right)\left(\intop_{G}\chi_{T,\delta}\left(g\right)\Xi\left(g\right)dg\right)\\
 & \ll\left(\log\left(T\right)+1\right)^{2C_{2}}T^{-2}.
\end{align*}

To show pointwise bounds, we notice that if $\chi_{T,\delta}*\chi_{T,\delta}\left(g\right)= R$, then $\chi_{T+1,\delta}*\chi_{T+1,\delta}\left(g'\right)\gg R$, for $g'$ in an annulus of size similar to that of $g$, i.e., for $C^{-1}\n{g}_\delta \le \n{g'}_\delta\le C\n{g}_\delta$ for some $C>1$. This annulus is of measure $\asymp\n{g}_\delta^2$. Therefore, 
\[
\chi_{T,\delta}*\chi_{T,\delta}\left(g\right)\n{g}_\delta^2\Xi(g) \ll \intop_{G}\left(\chi_{T+1,\delta}*\chi_{T+1,\delta}\right)\left(g'\right)\Xi\left(g'\right)dg' \ll\left(\log\left(T\right)+1\right)^{2C_{2}}T^{-2},
\]
and the first bound follows by applying the lower bound of Equation~\eqref{eq:Xi bounds}.

The bound on $\chi_{T}$ follows from the bound on $\chi_{T,\delta}$ and the relation between them.
\end{proof}

Now consider the locally symmetric space $X_{q}=\Gamma_{0}\left(q\right)\backslash G/K$. As in the $\SL_2$ case, it has finite measure, and we will consider the space $L^2\left(X_q\right)$, with the natural $L^2$-norm.

We first discuss the spectral gap. We denote by $L_{0}^{2}\left(X_{q}\right)$ the functions in $L^2\left(X_q\right)$ of integral $0$. Since $\chi_{T}$ is bi-$K$-invariant and sufficiently nice, the function $\chi_{T}$ acts by convolution from the right on
$f \in L^{2}\left(X_{q}\right)$, and the resulting function is well defined pointwise if $f$ is bounded. The operation sends $L_0(X_q)$ to itself.

\begin{thm}[Spectral Gap]\label{thm:spectral gap}
There exists $\tau>0$ such that for $T>0$ the operator $\chi_{T}$
satisfies for every $f\in L_{0}^{2}\left(X_{q}\right),$ 
\[
\n{f*\chi_{T}}_{2}\ll T^{-\tau}\n{f}_2.
\]
\end{thm}
The theorem follows from explicit versions of property (T), or explicit versions of the mean ergodic theorem (e.g., \cite[Section~4]{ghosh2013diophantine}) which are actually true for all lattices in $G=\SL_3(\R)$ uniformly in $T$ and the lattice.
It is remarkable that the proof of Theorem~\ref{thm:spectral gap} is much simpler than the proof of Theorem~\ref{thm:selberg spectral gap}.

As in the $\SL_2$ case, we define for $x_0\in X_q$ the function $b_{T,x_0} (x)=\sum_{\gamma \in \Gamma_0(q)} \chi_T(\tilde{x}_0^{-1}\gamma x)$, where $\tilde{x}_0$ is any lift of $x_0$ to $G$. 

We have a map $\iota: \Gamma_{0}\left(q\right)\backslash \Gamma \to X_{q}$ defined by $\iota\left(\Gamma_{0}\left(q\right)x_{0}\right) = \Gamma_{0}x_{0}K\in X_{q}$. By a slight abuse of notation we write $ \iota\left(\Gamma_0(q)x_0\right) = \iota(x_0)$.

The map $\iota$ has fibers of bounded size (independently of $q$), and we may choose $T_0$ small enough so that $\iota(y) \ne \iota(y')$ implies that $b_{T_0,\iota(y)}$ and $b_{T_0,\iota(y')}$ have disjoint supports. In addition, $b_{T_0,\iota(y)}$ will have a bounded $L^2$-norm as a function in $L^2\left(X_q\right)$.

\begin{lem}\label{lem:unfolding}
For $f\in L^2(X_q)$ bounded, 
\[
\left\langle f , b_{T,x_0} \right \rangle = \left(f*\chi_T^*\right)(x_0).
\]
\end{lem}
The proof is the same as the proof of Lemma~\ref{lem:unfolding SL2}.

\begin{lem} \label{lem:lattice count to L2}
Let $C>0, \epsilon_0>0$ fixed. Let $x_{0}\in\Gamma_{0}\left(q\right)\backslash \Gamma$
and assume for
$T'\le Cq^{2}$,
\[
|\left\{ \gamma\in \Gamma:\n{\gamma}_\delta\le T',x_{0}\gamma=x_{0}\right\}| \ll_{\epsilon_{0}}q^{\epsilon_{0}}T^{\prime}.
\]
Then there exists $C'>0$ depending only on $C$ such that for $T=C'q^{1/3}$ it holds that
for every $\epsilon>0$,
\[
\n{b_{T,\iota(x_{0})}}_{2}\ll_{\epsilon_{0},\epsilon}q^{-1+\epsilon_{0}+\epsilon}.
\]
\end{lem}

\begin{proof}
Notice that $\gamma\in \Gamma$ satisfies $\Gamma_{0}\left(q\right)x_{0}\gamma=\Gamma_{0}\left(q\right)x_{0}$
if and only if $\gamma\in x_{0}^{-1}\Gamma_{0}\left(q\right)x_{0}$ (the last group is a well defined subgroup of $\Gamma$).
Therefore we may rewrite the assumption in the following manner: For every $T'\le Cq^{2}$,
\begin{equation}
|\left\{ \gamma\in\Gamma_{0}\left(q\right):\n{x_{0}^{-1}\gamma x_0}_\delta\le T'\right\}| \ll_{\epsilon_{0}}q^{\epsilon_{0}}T^{\prime},\label{eq:lattice point count assumption}
\end{equation}
where we identify $x_0$ with a fixed element of  $\Gamma\le G$.

Write using Lemma~\ref{lem:unfolding}, 
\begin{align*}
\n{b_{T,\iota(x_0)}}_{2}^{2} & =\left\langle b_{T,\iota(x_0)},b_{T,\iota(x_0)}\right\rangle \\
 & =b_{T,\iota(x_{0})}*\chi_{T}^*(\iota(x_0)) \\
 & = \sum_{\gamma\in \Gamma_0(q)} (\chi_T*\chi_T^*)\left(x_0^{-1}\gamma x_0\right) \\
 & \ll_{\epsilon} T^{\epsilon} \psi_{C_{1}T^{6}} \left(x_0^{-1}\gamma x_0\right) ,
\end{align*}
where in the last inequality we used Lemma~\ref{lem:convolution lemma}.

Therefore, the lemma
will follow if we will prove that for $T=C'q^{1/3}$,
\begin{align*}
\sum_{\gamma\in\Gamma_{0}}\psi_{C_{1}T^{6}}\left(x_0^{-1}\gamma x_0\right)& =
T^{-6}\sum_{\gamma\in\Gamma_{0}\left(q\right):\n{x_{0}^{-1}\gamma x_{0}}_\delta\le C_1T^{6}}\n{x_{0}^{-1}\gamma x_{0}}_\delta^{-1} \\
& \ll
q^{-2}\sum_{\gamma\in\Gamma_{0}\left(q\right):\n{x_{0}^{-1}\gamma x_{0}}_\delta\le C_2q^2}\n{x_{0}^{-1}\gamma x_{0}}_\delta^{-1} \\
& \overset{!}{\ll}_{\epsilon}q^{-2+\epsilon_{0}+\epsilon},
\end{align*}
where $C_{2}=C_{1}C'^{6}$.

So it suffices to show that 
\[
\sum_{\gamma\in\Gamma_{0}\left(q\right):\n{x_{0}^{-1}\gamma x_{0}}_\delta^{-1}\le C_{2}q^{2}}\n{x_{0}^{-1}\gamma x_{0}}_\delta\overset{!}{\ll}_{\epsilon,\epsilon_0}q^{\epsilon+\epsilon_0}.
\]

We now apply Equation~\eqref{eq:partial summation}
(discrete partial summation), with $g\left(\gamma\right)=\n{\gamma}$,
$f\left(x\right)=x^{-1}$ and $Y=C_{2}q^{2}$ we have
\begin{align*}
\sum_{\gamma\in\Gamma_{0}\left(q\right):\n{x_{0}^{-1}\gamma x_{0}}_\delta\le C_{2}q^{2}}\n{x_{0}^{-1}\gamma x_{0}}_\delta^{-1}  \ll&|\left\{ \gamma:\n{x_{0}^{-1}\gamma x_{0}}_\delta\le C_{2}q^{2}\right\}| q^{-2}\\
 & +\intop_{1}^{C_{2}q^{2}}|\left\{ \gamma:\n{x_{0}^{-1}\gamma x_{0}}_\delta\le S\right\}| S^{-2}dS.
\end{align*}

Choosing $C'$ small enough so that $C_{2}=C_{1}C'^{6}\le C$ and
applying Equation~\eqref{eq:lattice point count assumption} we have that the last value satisfies
\begin{align*}
 & \ll_{\epsilon,\epsilon_0}q^{\epsilon+\epsilon_0}+q^{\epsilon+\epsilon_0}\intop_{1}^{C_{3}q^{2}}S^{-1}dS\\
 & \ll_{\epsilon}q^{2\epsilon+\epsilon_0},
\end{align*}
as needed.
\end{proof}

We denote by $\pi\in L^2(X_q)$ the constant probability function on $X_q$. 

Using the counting result Theorem~\ref{thm:lattice point} we will now show that for many points $x_0 \in \Gamma_0(q)\backslash \Gamma$ the condition of Lemma~\ref{lem:lattice count to L2} holds, and thus obtain:

\begin{lem}
\label{lem:main spectral lemma}There exists $C>0$, $\tau>0$, such that for every $\epsilon_{0}>0$, as $q\to\infty$ among primes, there exists a set $Y\subset\Gamma_{0}\left(q\right)\backslash \Gamma=P_{q}^{tr}$
of size $\left|Y\right|\ge\left(1-o_{\epsilon_{0}}(1)\right)\left|\Gamma_{0}\left(q\right)\backslash \Gamma\right|$,
such that for every $\Gamma_{0}x_{0}\in Y$, it holds
for $T=Cq^{1/3}$ that 
\[
\n{b_{T,\iota(x_{0})}*\chi_{T^\eta}-\pi}_2\ll_{\epsilon_{0}}q^{-1-\eta\tau+\epsilon_{0}}.
\]
\end{lem}

\begin{proof}
By Theorem~\ref{thm:lattice point} and Equation~\eqref{eq:equivalent distances}
it holds that for some $C>0$, for all $T\le Cq^{2}$ and $\epsilon>0$
\[
\sum_{x_{0}\in\Gamma_{0}\left(q\right)\backslash \Gamma}|\left\{ \gamma\in \Gamma:\n{\gamma}_\delta\le T,x_{0}\gamma=x_{0}\right\}| \ll_{\epsilon}q^{2+\epsilon}T.
\]

Since $\left|\Gamma_{0}\left(q\right)\backslash \Gamma\right|=\left(1+o\left(1\right)\right)q^{2}$,
we may choose a subset $Y\subset\Gamma_{0}\left(q\right)\backslash \Gamma$
of size 
\[
\left|Y\right|\ge\left(1-o_{\epsilon_{0}}(1)\right)\left|\Gamma_{0}\left(q\right)\backslash \Gamma\right|,
\]
such that for every $x_{0}\in Y$, 
\[
|\left\{ \gamma\in \Gamma:\n{\gamma}_\delta\le T,x_{0}\gamma=x_{0}\right\}| \ll_{\epsilon_{0}}q^{\epsilon_{0}}T.
\]
We now apply Lemma~\ref{lem:lattice count to L2} to every $x_{0}\in Y$
to obtain 
\[
\n{b_{T,\iota(x_{0})}}_2\ll_{\epsilon_{0}}q^{-1+\epsilon_{0}}.
\]

Next, we apply Theorem~\ref{thm:spectral gap} as in Lemma~\ref{lem:almost equidist L2 SL2} to deduce the final result.
\end{proof}

We may now finish the proof of Theorem~\ref{thm:main thm - alt version}, similar to the $\SL_2$ case. First,

\begin{lem}\label{lem:lifting from spectral} There is $C'>0$ such that for  $x_0,y\in \Gamma_0(q) \backslash \Gamma$, if  $\left\langle b_{T,\iota(x_0)}*\chi_{T^\eta},b_{T_0,\iota(y)}\right\rangle>0$, then there is $\gamma \in \Gamma$ such that $x_0\gamma=y$, and $\n{\gamma}_K \le C' T^{1+\eta}$.
\end{lem}

\begin{proof}
The proof is essentially the same as Lemma~\ref{lem:lifting from spectral SL2}. 
We have by Lemma~\ref{lem:unfolding} 
\[
b_{T,\iota(x_0)}*\chi_{T^\eta}*\chi_{T_0}^*(\iota(y_0))>0.
\]

Denote by $\tilde{x}_0,\tilde{y}$ as some lifts of $x_0,y$ to $\Gamma$. We get $g_1,g_2,g_3\in G$, $\gamma\in \Gamma_0(q)$ such that $\gamma^{-1} \tilde{x}_0 g_1g_2g_3=\tilde{y}$, with $g_1\in \supp(\chi_T)$, $g_2\in \supp(\chi_{T^\eta})$, $g_3 \in \supp\left(\chi_{T_0}^*\right)$. Writing $g = g_1g_2g_3$, we have that \[
\n{g}_K\ll \n{g_1}_K\n{g_2}_K\n{g_3}_K \ll T^{1+\eta}.
\]

In addition $g = \tilde{x}_0^{-1}\gamma \tilde{y} \in x_0^{-1}\Gamma_0(q) y \subset \Gamma$, which says that $x_0 \gamma =y$, as needed.
\end{proof}

To complete the proof, fix $\epsilon>0$. Let $x_0\in \Gamma_0(q) \backslash \Gamma$ be in the set $Y$ of Lemma~\ref{lem:main spectral lemma}.
Denote by $\tilde{Z}_{x_0}$ the set of elements $y\in \Gamma_0(q) \backslash \Gamma$ for which there is no $\gamma \in \Gamma$ with $\n{\gamma}_K \le q^{1/3+\epsilon}_K$ such that $x_0 \gamma =y$. It is enough to prove that $\tilde{Z}_{x_0} = o\left(\left|\Gamma_0(q) \backslash \Gamma\right|\right) = o\left(q^2\right)$. 

Choose $T=C q^{1/3}$, and $\eta$ small enough so that $C' T^{1+\eta} < q^{1/3+\epsilon}$, with $C$ as in Lemma~\ref{lem:main spectral lemma} and $C'$ as in Lemma~\ref{lem:lifting from spectral}. 

We denote $B = \sum_{y\in \tilde{Z}_{x_0}}b_{T,\iota(y)} \in L^2\left(X_q\right)$. Then by Lemma~\ref{lem:lifting from spectral}
\[
\left\langle
b_{T,x_0}*\chi_{T^\eta}-\pi, B
\right\rangle
= \frac{\left|\tilde{Z}_{x_0}\right|}{\text{Vol}(X_q)}\gg \frac{\left|\tilde{Z}_{x_0}\right|}{q^2}.
\]

On the other hand, by the choice of $x_0$ and Lemma~\ref{lem:main spectral lemma}, 
\begin{align*}
\left\langle
b_{T,x_0}*\chi_{T^\eta}-\pi, B
\right\rangle
&\ll
\n{B}_2 \n{b_{T,x_0}*\chi_{T^\eta}-\pi}_2 \\
& \ll_{\epsilon_0} \sqrt{\left|\tilde{Z}_{x_0}\right|} q^{-1-\eta\tau+\epsilon_0}.
\end{align*}

By combining the two estimates and choosing $\epsilon_0$ small enough, we get the desired result
\[
\left|\tilde{Z}_{x_0}\right| \ll_{\epsilon_0} q^{2-2\eta\tau-2\epsilon_0}=o\left(q^2\right).
\]

\section{Optimal Lifting for the Action on Flags}\label{sec:flag action}
In this section we prove optimal lifting for another action of $\SL_3(\Z)$. Let $B_q$ be the set of complete flags in $\F^3$, i.e.,
\[
B_q = \left\{\left(V_1,V_2\right) : 0 < V_1 < V_2 < \mathbb{F}_q^3\right\},
\]
i.e., $V_1\subset V_2$ are subspaces of $\F^3$, such that $\dim V_1=1$, $\dim V_2=2$.

There is a natural action action $\Phi_{q}\colon\SL_{3}\left(\Z\right)\to \Sym\left(B_{q}\right)$. It gives rise to a non-principal congruence subgroup
\[
\Gamma_{2}^{\prime}\left(q\right) =\left\{ \left(\begin{array}{ccc}
* & a & b\\
* & * & c\\
* & * & *
\end{array}\right)\in \SL_{3}\left(\Z\right):a=b=c=0\mod q\right\}.
\]
Concretley,
\[
\Gamma_{2}^\prime\left(q\right) = \left\{ \gamma\in \SL_{3}\left(\Z\right):\Phi_{q}\left(\gamma\right)(\boldsymbol{1})=\boldsymbol{1}\right\},
\]
where 
\[
\boldsymbol{1} = \left(\spann\{\left(\begin{smallmatrix}0\\0\\1\\\end{smallmatrix}\right)\},\spann\{\left(\begin{smallmatrix}0\\0\\1\end{smallmatrix}\right),\left(\begin{smallmatrix}0\\1\\0\end{smallmatrix}\right)\}\right)
\]

The result reads as follows:
\begin{thm}\label{thm:almost radius_flags}Let $\Gamma=\SL_{3}\left(\Z\right)$, and
for a prime $q$ let $B_q$ and $\Phi_{q}:\SL_{3}\left(\Z\right)\to\Sym\left(B_{q}\right)$ as above. Then for every $\epsilon>0$, as $q\to\infty$, there exists a set $Y\subset B_{q}$ of size $\left|Y\right|\ge\left(1-o_{\epsilon}(1)\right)\left|B_{q}\right|$,
such that for every $x\in Y$, there exists a set $Z_{x}\subset B_{q}$
of size $\left|Z_{x}\right|\ge\left(1-o_{\epsilon}(1)\right)\left|B_{q}\right|$,
such that for every $y\in Z_{x}$, there exists an element $\gamma\in\Gamma$
satisfying $\n{\gamma}_{\infty}\le q^{1/2+\epsilon}$, such
that $\Phi_{q}\left(\gamma\right)x=y$.
\end{thm}

The exponent $1/2$ is optimal, since the size of $B_{q}$ is
$\left|B_{q}\right|\asymp q^{3}$, while the number of elements $\gamma\in \SL_{3}\left(\Z\right)$
satisfying $\n{\gamma}_{\infty}\le T$ is $\asymp T^{6}$. 
This also hints why handling flags is harder than handling the projective plane: The volume of the homogenous space is larger ($q^3$ instead of $q^2$).
In comparison, the principal congruence subgroup gives the much larger volume $q^8$, and optimal lifting for it is still open.

The proof of Theorem~\ref{thm:almost radius_flags} is very similar to the proof of Theorem~\ref{thm:almost radius}.
The analytic part is essentially identical to Section~\ref{sec:Second Section}, with some minor modifications coming from the fact that the size $|P_q|\asymp q^2$ is replaced by $|B_q|\asymp q^3$. We therefore leave it to the reader.

The counting part needs a slightly more delicate argument. The needed result is an analog of Theorem \ref{thm:lattice point}, as follows:

\begin{thm}
\label{thm:lattice point flags}There exists a constant $C>0$ such that for every prime
$q$, $T\le Cq^{3}$ and $\epsilon>0$ it holds that 
\[
\left|\left\{ \left(\gamma,x\right)\in \SL_{3}\left(\Z\right)\times B_q:\n{\gamma}_{\infty}\|\gamma^{-1}\|_{\infty}\le T,\Phi_{q}\left(\gamma\right)\left(x\right)=x\right\} \right|\ll_{\epsilon}q^{3+\epsilon}T.
\]
\end{thm}

We prove Theorem~\ref{thm:lattice point flags} in the rest of this section.

By dyadically dividing the range of $\n{\gamma}_{\infty}$ into $O\left(\log\left(T\right)\right)$
subintervals, it is enough to prove that there exists $C>0$ such that for
every $S\leq R$ and $RS \leq Cq^{3}$:
\[
\left|\left\{ \left(\gamma,x\right)\in \SL_{3}\left(\Z\right)\times B_q :\n{\gamma}_{\infty}\le S,\|\gamma^{-1}\|_{\infty}\le R, \Phi_{q}\left(\gamma\right)\left(x\right)=x \right\}\right| \ll_\epsilon q^{3+\epsilon}RS 
\]

We divide into several cases according to the Jordan form of $\gamma \mod q$. We leave the verification of the following to the reader:
\begin{enumerate}
    \item If $\gamma = I  \mod q$ then there are $|B_q|\asymp q^3$ different $x \in B_q$ such that $\Phi_q(\gamma)x = x$.
    \item If $\gamma \mod q$ has an eigenspace of dimension $2$, but is not the identity, i.e., the Jordan form of $\gamma \mod q$ is \[
    \left(\begin{array}{ccc}
1 & 1 & 0\\
0 & 1 & 0\\
0 & 0 & 1
\end{array}\right)
\text{ or }
\left(\begin{array}{ccc}
\alpha & 0 & 0\\
0 & \alpha & 0\\
0 & 0 & \alpha^{-2}
\end{array}\right),
\]
then there are $\asymp q$ different $x\in B_q$ such that $\Phi_q(\gamma)x = x$. As in Section~\ref{sec:First Section}, we call such $\gamma$ bad mod $q$.
\item For all other cases, there are $O(1)$ different $x\in B_q$ such that $\Phi_q(\gamma)x = x$.
\end{enumerate} 

Theorem~\ref{thm:lattice point flags} will therefore follow from the following two lemmas:
\begin{lem}\label{lem:flag counting principal}
There exists $C>0$ such that for
every $S\leq R$ and $RS \leq Cq^{3}$:
\[
\left|\left\{ \gamma\in \SL_{3}(\Z) :\n{\gamma}_{\infty}\le S,\|\gamma^{-1}\|_{\infty}\le R, \gamma = I \mod q \right\}\right| \ll_\epsilon q^\epsilon RS.
\]
\end{lem}

\begin{lem}\label{lem:flag counting bad}

There exists $C>0$ such that for
every $S\leq R$ and $RS \leq Cq^{3}$:
\[
\left|\left\{ \gamma\in \SL_{3}(\Z) :\n{\gamma}_{\infty}\le S,\|\gamma^{-1}\|_{\infty}\le R, \gamma \operatorname{\ bad\ mod }q \right\}\right| \ll_\epsilon q^{2+\epsilon} RS.
\]
\end{lem}

\begin{proof}[Proof of Lemma~\ref{lem:flag counting principal}]
We will actually work a little harder than necessary to show that the count is at most $\ll q^\epsilon (S/q+1)^2$, which is tight up to $q^\epsilon$. 

Using ideas of \cite[Chapter 6]{heath2015spectral}, 
since for all $x$ it holds that
$\gamma - (1+xq)I = 0 \mod q$, it follows that \[\det(\gamma-(1+xq)I)=1-(1+xq)\tr\gamma^{-1}+(1+xq)^2\tr\gamma-(1+xq)^3=0 \mod q^3.\]
By equating coefficients, one obtains:
\begin{align*}
\tr \gamma &=\tr \gamma^{-1} = 3 \mod q^2\\
\tr\gamma &=\tr \gamma^{-1} \mod q^3
\end{align*}
In addition, 
$(\gamma-I)^2 = 0\mod q^2$, so 
$\gamma^{-1}= 2I-\gamma \mod q^2$.
We may assume that $\n{\gamma^{-1}}_\infty<q^2/2$,
thus 
$\gamma^{-1}=2I-\gamma$
and
$(\gamma-I)^2 = 0$.

Let us use the above information in order to parametrize the diagonals of $\gamma, \gamma^{-1}$. As in Section~\ref{sec:First Section}, we denote the entries of $\gamma$ by $a_{ij}$ and the entries of $\gamma^{-1}$ by $b_{ij}$.

Write
$a_{11} = 1+qa$, $a_{22} = 1+qb$, $a_{33}=1-q(a+b)$,
$b_{11} = 1-qa$, $b_{22} = 1-qb$, $b_{33} = 1+q(a+b)$. Then:
\begin{align*}
\gamma &= \left(\begin{array}{ccc}
1+qa & * & *\\
* & 1+qb & *\\
* & * & 1-(a+b)q
\end{array}\right)
\\
\gamma^{-1} &= \left(\begin{array}{ccc}
1-qa & * & *\\
* & 1-qb & *\\
* & * & 1+(a+b)q
\end{array}\right)
\end{align*}

The non-exceptional case happens if $a,b,a+b$ are all non-zero. Then we shall recover all of $\gamma$, up to $q^\epsilon$ options as follows. It holds that:
\[
a_{12}a_{21} = a_{11}a_{22}-b_{33}= 1+(a+b)q+q^2ab - (1+(a+b)q)=q^2ab\neq 0
\]
so using the divisor bound, we recover $a_{12},a_{21}$ up to $q^{\epsilon}$ options, and similarly for the other entries.

If all three $a,b,a+b$ vanish, then among any pair $1\le i<j\le3$, $a_{ij}a_{ji}=0$. There are $3$ such pairs, so we recover $\gamma$ up to $\ll\left(S/q + 1\right)^3$ options.
As a matter of fact, one can improve this estimate: Either up to permutations $\gamma$ is upper triangular, or it has a single non-diagonal non-zero contribution to the determinant. The diagonal contributes 1 to the determinat, so there cannot be a non-diagonal contribution. Therefore, we may assume that $\gamma$ is upper triangular. We know that $(\gamma-I)^2=0$, so $a_{12}a_{23}=0$. Thus there actually only $\ll\left(S/q + 1\right)^2$ options for $\gamma$.

For the remaining case, we may thus assume $a+b=0$, so $a=-b\ne 0$. Hence, $\gamma$ and $\gamma^{-1}$ are of the form:
\begin{align*}
\gamma &= \left(\begin{array}{ccc}
1+qa & cq & *\\
dq & 1-qa & *\\
* & * & 1
\end{array}\right)
\\
\gamma^{-1} &= \left(\begin{array}{ccc}
1-qa & -cq & *\\
-dq & 1+qa & *\\
* & * & 1
\end{array}\right).
\end{align*}

There are $\ll(S/q+1)$ options for $a$. 

It holds that
$cdq^2=a_{12}a_{21}=a_{11}a_{22}-b_{33}=1-q^2a^2-1$, 
hence
$cd=-a^2$.
Since $a\neq 0$ we obtain $c,d$ from $a$ up to $q^{\epsilon}$.

Next, let us note that:
$a_{31}a_{13}=a_{11}a_{33}-b_{22}=0$, and similarly $a_{23}a_{32}=0$. 

If both $a_{13},a_{31}=0$, there are $\ll (S/q+1)$ options for $a_{23},a_{32}$.

Otherwise, we may assume that $a_{13}\ne0$, and then $a_{31}=0$. Multiply the first row of $\gamma$ with the second column of $\gamma^{-1}$ to obtain
\[a_{11}b_{12}+a_{12}b_{22}+a_{31}b_{32}=0=(1+qa)(-cq)+cq(1+qa)+a_{13}b_{32},\]
that is $a_{13}b_{32}=0$, so $b_{32}=-a_{32}=0$. 

Multiply the first row of $\gamma$ with the third column of $\gamma^{-1}$ to show that $a_{13}$ determines $a_{23}$.

All in all, this brings us to $\ll q^\epsilon(S/q+1)^2$ options for $\gamma$, as needed.
\end{proof}

For the proof of Lemma~\ref{lem:flag counting bad} we will need the following:
\begin{lem}\label{lem:flag counting sublemma}
The number of solutions for Equations~\eqref{eq: trace relation}, \eqref{eq: trace relation mod q2} in $\tr \gamma,\tr \gamma^{-1} \in \Z$, $\alpha \in \F$, $|\tr\gamma|\le S$, $|\tr \gamma^{-1}| \le R$ is bounded by $\ll (S/q+1)(R/q+1)+q$.
\end{lem}

\begin{proof}
Assume that $(x_1,y_1,\alpha),(x_2,y_2,\alpha)$ are solutions. Then by Equation~\eqref{eq: trace relation}, $x_1-x_2=y_1-y_2=0 \mod q$.  Denote $z = (x_1 -y_1)/q$, $w = (x_2 -y_2)/q$. Notice that $|z|\le 2S/q$, $|w| \le 2R/q$. By Equation~\eqref{eq: trace relation mod q2} $(z, w,\alpha)$ is a solution to
$\alpha qz -qw = 0 \mod q^2$, or 
\begin{equation}\label{eq:simple eq}
\alpha z-w = 0\mod q.
\end{equation}

Therefore, $A$ solutions with the same $\alpha\in\F$ for \eqref{eq: trace relation},\eqref{eq: trace relation mod q2} give $A$ solutions to \eqref{eq:simple eq} with the same $\alpha\in\F$. 
So the total number of solutions is bounded by the number of solutions of 
Equation~\eqref{eq:simple eq} with $|z|\le2S/q$,$|w|\le 2R/q$, $\alpha\in\F$. The last number is bounded by $\ll (S/q+1)(R/q+1)+q$, since every choice of $z,w$ sets $\alpha$ uniquely, unless $z=w=0$. 
\end{proof}

\begin{proof}[Proof of Lemma~\ref{lem:flag counting bad}]
By Lemma~\ref{lem:flag counting sublemma} there are $\ll (S/q+1)(R/q+1)+q$ options for $\tr\gamma,\tr\gamma^{-1},\alpha$. In our range of parameters it holds that $RS\le Cq^3$ and since $\n{\gamma^{-1}}_\infty \le 2 \n{\gamma}$, we may assume that $R\le 2S^2$, so $R\ll q^2$, and therefore $(S/q+1)(R/q+1)+q\ll q$.

There are at most $S^2$ options for $a_{11},a_{22}$, and knowing $\tr{\gamma}$, we have now all of the diagonal of $\gamma$.
By Equation~\eqref{eq: relation mod q}, the diagonal of $\gamma$ determines the diagonal of $\gamma^{-1} \mod q$.
Lifting, the first two entries $b_{11},b_{22}$ have just $(R/q+1)^2$ options, giving $b_{33}$ for free.
Thus there are at most $\ll qS^2(R/q+1)^2$ options.

In the non-exceptional case when the non-diagonal entries are non-zero, the rest of the matrix has $\ll_\epsilon q^\epsilon$ options. So we should show that  
\[
qS^2(R/q+1)^2 \ll RSq^2,
\]
or $S(R/q+1)^2\ll Rq$. 
For $R<q$, this reduces to $S\ll Rq$, which is obvious. For $R>q$, this reduces to $RS\ll q^3$, which is again true.

Let us deal with the exceptional case. Without loss of generality we may assume that $a_{11}a_{22}=b_{33}$ and $a_{21} = 0$. 
We further separate into cases:
\begin{enumerate}
    \item If all other non-diagonal entries besides $a_{21}$ and $a_{12}$ are non-zero, then we may guess the diagonal of $\gamma$ and $\gamma^{-1}$ as before, and get the other non-diagonal entries using divisor bounds. The matrix $\gamma$ is then of the form 
    \[\left(\begin{array}{ccc}
         * & ? & \times \\
         0 & * & \times \\
         \times & \times & *
    \end{array} \right),
    \]
    with $a_{12}$ the only unknown and where $\times$ denotes a non-zero value. Then we get that $\det \gamma = Ea_{12}+F$, with $E=a_{23}a_{31}\ne0,F$ known, so $a_{12}$ is determined uniquely from $\det\gamma=1$.
    
    \item If $a_{31}=0$, then $a_{11} = \alpha = \pm 1$, and the matrix is of the form:
    \[\left(\begin{array}{ccc}
         \pm 1 & * & * \\
         0 & * & * \\
         0 & * & *
    \end{array}
    \right).\]  
    As in the first exceptional case of Section~\ref{sec:First Section}, denote $A=\left(\begin{array}{cc}
    a_{22} & a_{23}\\
    a_{32} & a_{33}
    \end{array}\right)$. We know that $\det A = \alpha = \pm1$, and either   $\tr A=0\mod q$ or $\tr A=2\mod q$. Therefore,
    $a_{22},a_{33}$ have at most $\ll S(S/q+1)$ options. If $a_{22}a_{33}\ne \det A = \pm 1 $ then we get $q^\epsilon$ options for $a_{23},a_{32}$ by the divisor bound. If $a_{22}a_{33}= \det A=\pm 1$, then they are both $\pm1$, and $a_{23}a_{32}=0$, so there are $\ll S$ options for $A$. So in any case $A$ has at most $q^{\epsilon}S(S/q+1)$ options.
    The remaining two entries have at most $S^2$ options, so all in all there are $S^3(S/q+1)$ options. It remains to prove that:
    \[
    S^3(S/q+1) \ll RSq^2,
    \]
    which is a simple verification.

    \item If $a_{23}=0$ then
    $a_{22}=\alpha =\pm1$, and the matrix is of the form 
    \[\left(\begin{array}{ccc}
         * & * & * \\
         0 & \pm1 & 0 \\
         * & * & *
    \end{array}
    \right).\]   
    We reduce to the previous case (after permuting indices and transposing).
    
    \item We may now assume $a_{31}\ne0$, $a_{23}\ne0$. If $a_{13}=0$, we may assume $a_{12}\ne 0$, otherwise we reduce to a previous case.
    We now guess the diagonals as before, and further diverge into subcases: 
    \begin{enumerate}
        \item If $a_{32}\ne0$: Then since $a_{23}\ne0$ we have $a_{23}a_{32}=a_{22}a_{33}-b_{11}$, so we have $\ll_\epsilon q^\epsilon$ options for $a_{23},a_{32}$ by the divisor bound. Then the matrix is of the form:
    \[\left(\begin{array}{ccc}
         * & ? & 0 \\
         0 & * & \times \\
         ? & \times & *
    \end{array}
    \right).\]    
        From $\det\gamma=1$ we get $a_{12}a_{31}$, which is non-zero. By the divisor bound we are done.
        \item If $a_{32}=0$, the matrix is of the form: 
        \[\left(\begin{array}{ccc}
         * & ? & 0 \\
         0 & * & ? \\
         ? & 0 & *
    \end{array}
    \right).\]
        From $\det\gamma=1$ we get $a_{12}a_{23}a_{31}$, which is again non-zero, and by the divisor bound we are done.
    \end{enumerate}
    \item 
    If $a_{13}\ne0$, $a_{23}\ne0$, $a_{31}\ne0$, $a_{32}=0$. We may assume that $a_{12}\ne0$ otherwise we reduce to a previous case. Then we guess the diagonals as usual, and since $a_{31}a_{13}\ne 0$ we know them in $\ll_\epsilon q^\epsilon$ ways by the divisor bound. Then the matrix is of the form:
    \[\left(\begin{array}{ccc}
         * & ? & \times \\
         0 & * & ? \\
         \times & 0 & \alpha ^{-2}
    \end{array}
    \right).\]
    From $\det\gamma=1$ we get $a_{12}a_{23}$ which is non-zero, and by the divisor bound we are done.
\end{enumerate}

\end{proof}

\bibliographystyle{acm}
\bibliography{./database}

\end{document}